\documentclass[a4paper]{article}

\usepackage[english]{babel}

\usepackage[utf8]{inputenc}
\usepackage[utf8]{inputenc}
\usepackage{amsmath}
\usepackage{graphicx}
\usepackage{amssymb}
\usepackage{amsthm}
\usepackage{tikz-cd}
\usepackage{mathrsfs}
\usepackage[colorinlistoftodos]{todonotes}
\usepackage{enumitem}
\usepackage{yfonts}
\usepackage{ dsfont }
\usepackage{MnSymbol}
\usepackage{slashed}
\usepackage{bbm}

\newcommand{\Cov}{{\rm Cov}}

\title{Gradient flow in the kernel learning problem}

\author{Yang Li\footnote{Cambridge University, Department of Pure Mathematics and Mathematical Statistics},~~~Feng Ruan\footnote{Northwestern University, Department of Statistics and Data Science}}

\date{\today}
\newtheorem{thm}{Theorem}[section]
\newtheorem{lem}[thm]{Lemma}

\theoremstyle{definition}
\newtheorem{eg}[thm]{Example}

\newtheorem{cor}[thm]{Corollary}

\newtheorem{rmk}{Remark}
\newtheorem{prop}[thm]{Proposition}
\newtheorem{Def}[thm]{Definition}
\newtheorem{Question}{Question}

\newtheorem*{Acknowledgement}{Acknowledgement}

\newcommand{\ie}{\emph{i.e.} }
\newcommand{\cf}{\emph{cf.} }

\newcommand{\R}{\mathbb{R}}
\newcommand{\C}{\mathbb{C}}

\newcommand{\norm}[1]{\left\lVert#1\right\rVert}

\newcommand{\E}{\mathbb{E}}

\DeclareMathOperator{\Hom}{Hom}
\DeclareMathOperator{\End}{End}

\DeclareMathOperator{\Tr}{Tr}

\begin{document}
	\maketitle

	\begin{abstract}
		This is a sequel to our paper `On the kernel learning problem'. We identify a canonical choice of Riemannian gradient flow, to find the stationary points in the kernel learning problem. In the presence of Gaussian noise variables, this flow enjoys the remarkable property of having a continuous family of Lyapunov functionals, and the interpretation is the automatic reduction of noise. 
		
		PS. We include an extensive discussion in the postcript explaining the comparison with the 2-layer neural networks. Readers looking for additional motivations are encouraged to read the postscript immediately following the introduction.
	\end{abstract}

	\section{Introduction}

	\subsection{The kernel ridge regression problem}

	Let $V$ be a real $d$-dimensional  vector space with an inner product $|\cdot|_V^2$, and let $V^*$ denote its dual space. Given a function $k_V: V^*\to \R^+$ with $\int_{V^*} k_V(\omega)d\omega=1$, we can define a Hilbert space of functions,
	\[
	\mathcal{H}=\{  f: V\to \C| \norm{f}_{\mathcal{H}}^2= \int_{V^*} \frac{ |\hat{f}|^2(\omega)}{ k_V(\omega) }d\omega<+\infty\}, 
	\]
	which embeds continuously into $C^0$, and in fact $\norm{f}_{C^0}\leq \norm{f}_{\mathcal{H}}$. Following \cite{Smale}\cite{Aronszajn},
	we say $\mathcal{H}$ is a \emph{reproducing kernel Hilbert space} (RKHS).

	Let $X$ be a random variable taking value in $V\simeq \R^d$, and $Y$ be a complex valued random variable with $\E[|Y|^2]<+\infty$. 
	We consider the following minimization problem, which we will refer to as \emph{kernel ridge regression} \cite{RuanLi}. Let $\lambda>0$ be a small fixed parameter. We define the \emph{loss function}
	\begin{equation}
		I(f, U, \lambda)= \frac{1}{2} \E[ |Y-f(UX)|^2] + \frac{\lambda}{2} \norm{f}_{\mathcal{H}}^2,
	\end{equation}
	depending on the parameter $U\in \End(V)$, and we seek the minimizer for
	\[
	J(U, \lambda)=  \inf_{f\in \mathcal{H}} I(f,U,\lambda).
	\]
	The existence of a unique minimizer among all $f\in \mathcal{H}$ follows from the Riesz representation theorem. By definition, the kernel ridge regression problem is a \emph{linear} problem in infinite dimensions. The minimizer $f_U$ for $I(f,U,\lambda)$ can be viewed as the optimal fit for $Y$ as a function of $X$, given the regularization effect imposed by the Hilbert norm term. By plugging in the zero test function, there is a `trivial upper bound' $J(U,\lambda)\leq \frac{1}{2}\E[|Y|^2]$.

	\subsection{Kernel learning problem}

	When $k_V$ is a \emph{rotationally invariant} function, namely $k_V(\omega)= k(|\omega|_V^2)$ for some positive valued function $k$, the kernel ridge regression problem has the following equivalent formulation. Given any (non-degenerate) inner product $\Sigma$ on $V$, the dual inner product\footnote{In index notation, if the inner product on $V$ is represented by $(g_{ij})$, then the dual inner product is represented by the inverse matrix $(g^{ij})$.} on $V^*$ is denoted as $\Sigma^{-1}$. We then introduce a Hilbert space of functions 
	\[
	\mathcal{H}_\Sigma= \{  F: V\to \C|   \norm{F}_{\mathcal{H}_\Sigma}^2
	=\sqrt{\det\Sigma}\int_{V^*} \frac{|\hat{F}|^2}{ k(|\omega|_{\Sigma^{-1}}^2) } d\omega \},
	\]
	and consider the following minimization problem with respect to $F$,
	\begin{equation}\label{eqn:functionalI2}
		\mathcal{I}(F, \Sigma, \lambda)=  \frac{1}{2} \E[|Y-F(X)|^2] + \frac{\lambda}{2} \norm{F}_{\mathcal{H}_\Sigma}^2,
	\end{equation}
	\[
	\mathcal{J}(\Sigma; \lambda):=\min_{F\in \mathcal{H}_\Sigma }  \frac{1}{2} \E[|Y-F(X)|^2] + \frac{\lambda}{2} \norm{F}_{\mathcal{H}_\Sigma}^2. 
	\]
	In this formulation, the functionals depend only on $\Sigma$, but not on the ambient inner product $|\cdot|_V^2$.

	Now for each $U\in GL(V)$, using the ambient inner product $|\cdot|_V^2$, we can define the inner product $\Sigma$ on $V$ and the dual inner product $\Sigma^{-1}$ by the formula
	\[
	|x|_\Sigma^2= |Ux|^2_V=(x, U^TUx)_V, \quad |\omega|_{\Sigma^{-1}}= \max_{|x|_\Sigma=1} \langle \omega, x\rangle,
	\]
	where $U^T$ denotes the adjoint operator of $U$.
	By a change of variables \cite[Lemma 4.1]{RuanLi}, the two formulations are related by
	\begin{equation}\label{eqn:twominimizationproblems}
		F= f\circ U,\quad   I(f, U, \lambda)= \mathcal{I}(F, \Sigma, \lambda),\quad J(U,\lambda)= \mathcal{J}(\Sigma; \lambda). 
	\end{equation}
	The minimizer $F_\Sigma=f_U\circ U$ gives the optimal fit for $Y$ as a function of a suitable linear transform of $X$, given the regularization effect imposed by the Hilbert norm term.

	We denote $Sym^2_+$ as the space of all non-degenerate inner products on $V$, which is an open submanifold inside the vector space $Sym^2 V^*$. In the \emph{kernel learning problem}, we view $\mathcal{J}(\Sigma;\lambda)$ as a functional on $Sym^2_+$, and we seek the local minimizers with respect to $\Sigma\in Sym^2_+$, referred to as \emph{vacua} in \cite{RuanLi}. We emphasize that unlike the kernel ridge regression problem, the dependence of $\mathcal{J}$ on $\Sigma$ is \emph{nonlinear}.

	Our motivation is that the vacua  reflect two kinds of information: variable selection (which components of $X$ does $Y$ depend on essentially), and scale detection (what are the natural scale parameters in the distribution of $X$).  Instead of simply learning $Y$ as a function of $X$ in an a priori fixed Hilbert space as in the kernel ridge regression problem, kernel learning searches among a family of Hilbert space norms parametrized by $Sym^2_+$, and aims to 
	find the landscape of vacua, and thereby find the most efficient Hilbert space norms adapted to the various regimes of the learning task. More substantial discussions can be found in the companion paper \cite{RuanLi}.

	\begin{rmk}\label{rmk:partialcompactification}
		In \cite{RuanLi}, it was explained that there is a partial compactification of $Sym^2_+$, where the vacua(=local minimizer of $\mathcal{J}(\Sigma; \lambda)$) naturally lives. Morally speaking, this reflects the fact that $\Sigma$ can limit to a degenerate inner product, or it can diverge to infinity while remaining bounded on a subspace. The first situation is relevant for variable selection tasks, and the second situation is sometimes relevant if some marginal distribution of $X$ contains discrete atoms.

		In this paper, for simplicity we only worry about the first possibility. 
		The degenerate inner products are stratified according to the dimension of the null subspace. The open strata are
		\[
		\mathcal{S}_l= \{  \Sigma \text{ is a positive semi-definite inner product on $V$ with rank $l$} \},
		\]
		for $l=0,1,\ldots d-1.$
		Each $\mathcal{S}_l$ is a smooth manifold. The space of positive semi-definite forms on $V$ is
		\begin{equation*}
			Sym^2_{\geq 0}= Sym^2_+\cup \bigcup_{l\leq d-1} \mathcal{S}_l ,
		\end{equation*}
		which is a closed subset of $Sym^2 V^*$. 
	\end{rmk}

	\subsection{The gradient flow}

	This paper is motivated by the \emph{dynamical question} of how to find the vacua of $\mathcal{J}(\Sigma;\lambda)$ algorithmically,  in terms of a \emph{gradient flow on the Riemannian manifold} $Sym^2_+$. We emphasize that the first variation of $\mathcal{J}$ is naturally a 1-form on $Sym^2_+$, and to define the \emph{gradient vector field}, one has to \emph{choose some Riemannian metric} on $Sym_+^2$; this choice is part of the design of the gradient flow. The readers interested in the relevance of this paper to 2-layer neural networks, are encouraged to refer to the postscript for some extra motivations and an informal survey of results.

	The main contribution of this paper is to identify a \emph{canonical choice of Riemannian metric}  designed so that when the random variable $X$ contains many noise Gaussian variable components, the Riemannian gradient flow turns out to possess \emph{a continuous family of Lyapunov functionals}, whose monotonicity reflects the decrease of the influence of the noise. The definition of our gradient flow only requires the input of the covariance matrix of $X$, involves no further additional choice, and is manifestly coordinate free.

	In the rest of the paper, we assume $X$ is a random variable valued in the vector space $V\simeq \R^d$, and $Y$ is a random variable valued in the complex numbers, with $\E[|X|^2]<+\infty$, and $\E[|Y|^2]<+\infty$. Thus we can define the \emph{covariance matrix} of $X$, which in coordinate free language is an inner product on $V^*$,
	\[
	\Cov(X)(\omega,\omega)= \E[\langle \omega, X- \E[X]\rangle^2].
	\]
	The dual inner product on $V$ is the inverse covariance matrix, which canonically gives an inner product on $V$, and we shall identify it with $|\cdot|_V^2$. In index notation, $\Cov(X)$ is represented by the matrix
	\[
	C^{ij}= \E[(X^i-\E X^i)( X^j-\E X^j) ] ,
	\]
	and the inverse covariance matrix is represented by the inverse matrix $C_{ij}$.

	We start with some basic geometry on $Sym^2_+$. At any point $\Sigma\in Sym^2_+$, the tangent space $T_\Sigma Sym^2_+$ is identified as $Sym^2 V^*$ (in index notation, this means a symmetric tensor with two lower indices), and the cotangent space $T_\Sigma^* Sym^2_+$ is identified as $Sym^2 V$ (in index notation, this means a symmetric tensor with two upper indices). We shall prescribe the Riemannian metric on the cotangent spaces first, which would induce a metric on the tangent spaces. Given any cotangent vectors $A, A'\in T_\Sigma^* Sym^2_+$, we define the Riemannian metric $\mathfrak{g}$ by the following contraction of tensors:
	\begin{equation}\label{eqn:Riemannianmetric}
		\mathfrak{g}(A, A')=\frac{1}{2} \Sigma_{ab} C_{cd} A^{ac} A'^{bd} + \frac{1}{2} \Sigma_{ab} C_{cd} A^{bd} A'^{ac}.
	\end{equation}

	The Riemannian metric allows us to identify the differential $D\mathcal{J}(\Sigma;\lambda)$ of the functional $\mathcal{J}$ (which is naturally a \emph{cotangent vector}) as a \emph{tangent vector} on the manifold $Sym^2_+$.
	In the index notation, we can write the gradient as
	\[
	(\nabla^{\mathfrak{g}} \mathcal{J}(\Sigma;\lambda) )_{ab} = \frac{1}{2}\Sigma_{ac} C_{bd} D\mathcal{J}(\Sigma;\lambda)^{cd} + \frac{1}{2}\Sigma_{bd} C_{ac} D\mathcal{J}(\Sigma;\lambda)^{cd} \in Sym^2 V^*.
	\]
	The \emph{downward gradient flow} on $Sym^2_+$ is  a flow of $\Sigma(t)$ according to the dynamics
	\begin{equation}\label{eqn:gradientflow}
		\frac{d}{dt} \Sigma(t)= - \nabla^{\mathfrak{g}} \mathcal{J}(\Sigma;\lambda).
	\end{equation}
	Tautologically the gradient flow will decrease the functional $\mathcal{J}$, namely
	\begin{equation}\label{eqn:Jdecreases}
		\frac{d}{dt} \mathcal{J}(\Sigma) (t)= - |\nabla^{\mathfrak{g}} \mathcal{J}(\Sigma;\lambda)|^2_{\mathfrak{g}}=  -  |D\mathcal{J}|^2_{\mathfrak{g}}.
	\end{equation}

	\begin{rmk}
		The manifold $Sym^2_+ $ carries several well known  Riemannian metrics. First, the open inclusion $Sym^2_+\subset Sym^2V^*$ induces the restriction of the Euclidean metric on $Sym^2V^*$. Second, the manifold $Sym^2_+$ can be identified as the symmetric space $GL_+(d)/SO(d)$, which carries a metric homogeneous under the $GL_+(d)$ action. Our choice of the Riemannian metric $\mathfrak{g}$ is \emph{not} one of these two metrics; we are not aware of its previous appearance in the literature.
	\end{rmk}

	This Riemannian gradient flow has an alternative characterization:

	\begin{thm}
		(See Proposition \ref{prop:Euclideanflow})  Under the map 
		\[
		\pi: \End(V)\to Sym^2_{\geq 0}, \quad U\mapsto \Sigma=U^T U,
		\]
		the \emph{Euclidean gradient flow} on $\End(V)$ for the functional $J(U,\lambda)$
		\begin{equation}\label{eqn:Euclideanflow}
			\frac{dU_a^i}{dt}= - \frac{1}{4} C_{ac} C^{ij} (DJ(U,\lambda))_j^c  =  - \frac{1}{4} C_{ac} C^{ij}  \frac{\partial J(U,\lambda)}{\partial U_j^c}  ,
		\end{equation}
		is a lift of the Riemannian gradient flow.
	\end{thm}

	This perspective exhibits closer analogy with more popular composition models such as neural networks, as we will discuss in the postscript.

	\subsection{Main theorem}

	We shall prove several foundational properties of the gradient flow (\ref{eqn:gradientflow}).

	\textbf{Assumptions}. The random variables satisfy $\E[|Y|^2]+\E[|X|^8]<+\infty$, and the kernel function  is rotationally invariant, $K(x)=\mathcal{K}(|x|_V^2)$, such that the derivative $\mathcal{K}', \mathcal{K}'': \R_{\geq 0}\to \R$ are bounded continuous functions, and $r\mathcal{K}'(r)$ is bounded for large $r$.

	These conditions guarantee  the first variation formula for $D\mathcal{J}$ (\cf Theorem \ref{thm:firstvariation} below), and that $\mathcal{J}$ extends as a $C^1$ function on $Sym^2_{\geq 0}$, and its first variation $D\mathcal{J}$ extends as a Lipschitz continuous tensor-valued function on $Sym^2 V^*$. 	These conditions are not intended to be optimal.

	\begin{thm}\label{thm:longtimeexistence}
		(See Sections \ref{sect:extensionflow},  \ref{sect:longtimeexistence}) Under the assumptions, 
		\begin{enumerate}
			\item   The Riemannian gradient flow  extends to a flow on  $Sym^2_{\geq 0}$, depending continuously on the initial data, and exists for all time $t\geq 0$.

			\item   If the initial data is contained in some open stratum $\mathcal{S}_l$ for $0\leq l\leq d-1$, then the flow stays inside $\mathcal{S}_l$.

		\end{enumerate}

	\end{thm}

	The next Theorem concerns the convergence behavior as $t\to +\infty$. We are interested in finding stationary points of $\mathcal{J}(\Sigma;\lambda)$. The main subtlety is that the infinite time limit may lie in some boundary stratum $\mathcal{S}_l\subset Sym^2_{\geq 0}$ (See Remark \ref{rmk:partialcompactification}). We shall say that a point $\Sigma_0\in \mathcal{S}_l$ is a \emph{stationary point}, if $D\mathcal{J}(\Sigma_0)$ vanishes on the tangent space of the stratum manifold $\mathcal{S}_l$ at $\Sigma_0$.

	\begin{thm}\label{thm:convergence}(See Section \ref{sect:convergence})
		In the setting of Theorem \ref{thm:longtimeexistence}, suppose further that $X$ is a continuous random variable. Then there is a subsequence of time $t_i\to +\infty$, such that $\Sigma(t_i)$ converge to some $\Sigma_0\in Sym^2_{\geq 0}$, which is a \emph{stationary point}.
	\end{thm}

	Furthermore, the gradient flow has a remarkable \emph{de-noising effect}. We suppose the vector space $V$ is given a direct sum decomposition \[
	V=W\oplus W',
	\]
	so the random variable $X$ has two components $(X_W, X_{W'})\in W\oplus W'$. We assume that $X_W$ is a \emph{Gaussian vector} $X_W\sim N(\E[X_W], \Cov(X_W)) $  independent of $X_{W'}$, and the random variable $\E[Y|X]$ depends only on $X_{W'}$. The distribution law of $X_{W'}$ is arbitrary, subject to the standing assumption that  $\E[|X|^2]<+\infty$.

	Given any unit vector $w\in W$ with respect to the fixed inner product $|\cdot|_V^2$,  we can compute the norm $|w|_{\Sigma(t)}$ with respect to the evolving inner products $\Sigma(t)$ on $V$, which measures the effect of the Gaussian noise.

	\begin{thm}\label{thm:montonicity}(See Theorem \ref{thm:Gaussiannoisemonotonicity}) In the setting of Theorem \ref{thm:longtimeexistence}, 
		for any unit vector $w\in W$, the functional $|w|_{\Sigma(t)}$ decreases in time along the gradient flow.
	\end{thm}

	In other words, the flow possesses a continuous family of Lyapunov functionals, which is a very strong property on an ODE system. This holds without any parametric assumptions on the distribution of $Y$ or the signal variables $X_{W'}$, beyond the moment bounds which guarantee the existence of the flows. 
	In this sense we view the choice of the Riemannian metric as optimal.

	\section{Background: RKHS and gradient formula}

	\subsection{Reproducing kernel function}

	We recall some basics of reproducing kernel Hilbert spaces (See \cite{RuanLi} for more details). 
	We define the kernel function
	\begin{equation}
		K(x)= \int_{V^*} e^{2\pi i \langle x, \omega\rangle} k_V(\omega) d\omega,
	\end{equation}
	and let $K(x, y)= K(x-y)$, so in particular $K(x,y)=\overline{K(y,x)}$. The functions $K(x,\cdot)$ lie in $\mathcal{H}$, and have Fourier transform $k_V(\omega) e^{2\pi i \langle \omega, x\rangle}$. These functions are known as reproducing kernels. 
	Each $f\in \mathcal{H}$ can be represented as 
	\begin{equation}\label{reproducingformula}
		f(x)= \langle f, K(x,\cdot)\rangle_{\mathcal{H}}.
	\end{equation}

	\begin{eg}\label{eg:Gaussiankernel}
		The Gaussian kernel corresponds to the case that $K(x)=e^{-\beta |x|_V^2}$ for some $\beta>0$. By taking the Fourier transform,
		\[
		k_V(\omega)= (\pi \beta^{-1})^{d/2} e^{- \pi^2 
			\beta^{-1} |
			\omega|^2}.
		\]
	\end{eg}
	
	\begin{eg}\label{eg:Sobolevkernel}
		(Sobolev kernel \cite[Section 4.5]{RuanLi}) The choice
		\[
		k_V(\omega)=k_{d,\gamma}(|\omega|^2)=\frac{(4\pi)^{d/2}\Gamma(\gamma+d/2 )}{  \Gamma(\gamma) } (1+|2\pi \omega|^2)^{-\gamma-d/2},
		\]
		which defines the Sobolev spaces $H^{d/2+\gamma}(\R^d)$. The $L^1$-integrability on $k_V$ is equivalent to $\gamma>0$. The Gamma function factors here appear as normalization constants. The reproducing kernel function is
		\[
		K(x)=  K_\gamma(|x|^2)=   \frac{  1  }{\Gamma(\gamma)} \int_0^{+\infty} y^{\gamma-1} e^{-y} e^{- \frac{ |x|^2}{4y} }dy .
		\]
		This depends only on $\gamma>0$, but not on $d$.
	\end{eg}

	\subsection{First variation formula}

	We are interested in how $\mathcal{J}(\Sigma;\lambda)$ varies with $\Sigma\in Sym^2_+$. The following first variation formula is proved in \cite[Section 4.2]{RuanLi}, and applies to a wide range of rotational invariant kernels, such as the Sobolev kernel and the Gaussian kernel.

	\begin{thm}(First variation)\label{thm:firstvariation}
		We denote $r_\Sigma(X,Y)=Y-F_\Sigma(X)= Y-f_U(UX)$, and let $(X',Y')$ be an independent copy of $(X,Y)$. 
		We suppose $\E[|Y|^2]<+\infty$, and the kernel function is $K(x)=\mathcal{K}(|x|^2_V)$ for some radial function
		$\mathcal{K}:\R_{\geq 0}\to \R_+$, such that {(i)} $\mathcal{K}'(r)$ is continuous on $\R_+$, and $r\mathcal{K}'(r)\to 0$ as $r\to 0$, (ii) either $|\mathcal{K}'(r) r|$ is bounded as $r \to \infty$ or $|\mathcal{K}'(r)|$ is monotone decreasing for large enough $r$, and (iii) the integrability hypothesis holds:
		\[
		\E[ |\mathcal{K}'(|X-X'|^2_V)|^2 |X-X'|_V^4   ] <+\infty.
		\]
		Then at any $\Sigma\in Sym^2_+$, 
		\begin{equation*}\label{eqn:Jfirstvariation3}
			\begin{split}
				D_\Sigma \mathcal{J}(\Sigma, \lambda)= &   -\frac{1}{2\lambda}   \E[  r_\Sigma(X,Y) \overline{ r_\Sigma(X',Y')}  D_\Sigma \mathcal{K}(|X-X'|_{\Sigma}^2) ] 
				\\
				= &   -\frac{1}{2\lambda}   \E[  r_\Sigma(X,Y) \overline{r_\Sigma(X',Y')}  \mathcal{K}'(|X-X'|_{\Sigma}^2) (X-X')\otimes (X-X') ] .
			\end{split}
		\end{equation*}
		
	\end{thm}

	\begin{Def}\label{def:C1ext}
		We say $\mathcal{J}(\Sigma;\lambda)$ extends to a $C^1$-function on $Sym^2_{\geq 0}$, if there is a differential $D\mathcal{J}$ on $Sym^2_{\geq 0}$ which is a $Sym^2 V$-valued  continuous function, such that at any $\Sigma,\Sigma'\in Sym^2_{\geq 0}$, we have
		\[
		\mathcal{J}(\Sigma',\lambda)= \mathcal{J}(\Sigma, \lambda)+ \langle D\mathcal{J}(\Sigma;\lambda),\Sigma'-\Sigma\rangle +o(|\Sigma'-\Sigma|).
		\]
	\end{Def}

	The following criterion for continuous extension is a special case of \cite[Proposition 4.7]{RuanLi} (See also \cite[Example 4.8]{RuanLi}). It applies to Sobolev kernels with $\gamma>1$, and to all Gaussian kernels.

	\begin{prop}\label{prop:continuousextensionJ}
		Suppose that $\mathcal{K}':\R_{\geq 0}\to \R$ is a bounded continuous function, $r\mathcal{K}'(r)$ is bounded for large $r$, and $\E[|X|^4]+ \E[|Y|^2]<+\infty$. 
		Then
		$\mathcal{J}(\Sigma;\lambda)$ extends to a $C^1$-function on  $Sym^2_{\geq 0}$.
	\end{prop}

	The following Corollary is needed for applying ODE existence results.

	\begin{cor}\label{cor:Lipextension}\cite[Corollary 4.10]{RuanLi}
		Suppose furthermore that $\mathcal{K}''$ is bounded, and $\E[|X|^8]<+\infty$. Then  $D\mathcal{J}$ extends to a Lipschitz continuous tensor valued function on $Sym^2 V^*$.
	\end{cor}

	\begin{rmk}
		As a caveat, for the Sobolev kernel, $\mathcal{K}$ only has a finite number of derivatives at zero, so we can only expect $\mathcal{J}(\Sigma;\lambda)$ to be differentiable to finite order on the boundary of $Sym^2_{\geq 0}$, instead of being smooth.
	\end{rmk}

	\section{Riemannian gradient flow}

	\subsection{Riemannian metric in the eigenbasis}\label{Sect:Riemannianmetriceigenbasis}

	We now seek to understand the Riemannian metric $\mathfrak{g}$ in the diagonal form. Given a point $\Sigma\in Sym^2_+$, we can simultaneously diagonalize $\Sigma$ and the given inner product $|\cdot|_V^2$ in an orthogonal basis $\{ e_i\}$ of $V$,
	\[
	C_{ij}= \delta_{ij}, \quad \Sigma_{ij}=\lambda_i \delta_{ij}. 
	\]
	Then given any $A\in T_\Sigma^* Sym^2_+\simeq Sym^2 V$, we can write $A= \sum A^{ij} e_i\otimes e_j$ for some symmetric matrix $A_{ij}$, and
	\[
	|A|_{\mathfrak{g}}^2= \Tr (A\Sigma A) =\sum_{i,j}  A^{ij} \lambda_jA^{ji} =\sum_{i,j} \lambda_i |A^{ij}|^2= \frac{1}{2} \sum_{i,j} (\lambda_i+\lambda_j) |A^{ij}|^2. 
	\]
	In terms of the tangent vectors $B\in T_\Sigma Sym^2_+\simeq Sym^2 V^*$, we can write $B=\sum B_{ij}e_i^*\otimes e_j^* $, and 
	\begin{equation}\label{eqn:Riemannianmetrictangentspace}
		|B|_{\mathfrak{g}}^2=  \sum_{i,j} \frac{2}{(\lambda_i+\lambda_j)} |B_{ij}|^2.
	\end{equation}
	Along the boundary strata $\mathcal{S}_l\subset Sym^2_{\geq 0}$, the Riemannian metric on the tangent bundle of $Sym^2_+$ blows up in the directions transverse to $\mathcal{S}_l$, while the Riemannian metric on the cotangent bundle of $Sym^2_+$ extends to a continuous tensor field on $Sym^2_{\geq 0}$.



	\subsection{Extension of the flow to $Sym^2_{\geq 0}$}\label{sect:extensionflow}

	Our goal is to extend the flow from $Sym^2_+$ to $Sym^2_{\geq 0}$, even though the Riemannian metric degenerates on the boundary of $Sym^2_{\geq 0}$.

	We assume the setting of Theorem \ref{thm:longtimeexistence}. By Corollary \ref{cor:Lipextension}, 
	$D\mathcal{J}$ extends to a Lipschitz  vector field on the vector space $Sym^2 V^*$, which we still denote as $D\mathcal{J}$. Then by Picard-Lindel\"of, the flow 
	\begin{equation}\label{eqn:gradientflow2}
		\frac{d\Sigma_{ab}}{dt}= - \frac{1}{2} (\Sigma_{ac} C_{bd} (D \mathcal{J})^{cd}   +  \Sigma_{bd} C_{ac}  (D \mathcal{J})^{cd}    )
	\end{equation}
	exists for some definite time for initial data in any given compact set in $Sym^2 V^*$, and depends continuously on the initial data.

	\begin{lem}
		Under the flow (\ref{eqn:gradientflow2}), if the initial data $\Sigma$ lies on the stratum $\mathcal{S}_l$, then the flow $\Sigma(t)$ stays in the stratum $\mathcal{S}_l$. Hence the flow preserves the semi-positive cone $Sym^2_{\geq 0}= Sym^2_+\cup \bigcup_l \mathcal{S}_l$, and in particular the restriction of the flow to $Sym^2_{\geq 0}$ does not depend on the extension of $D\mathcal{J}$.
	\end{lem}

	\begin{proof}
		We work in a fixed orthonormal basis of $|\cdot|_V^2$, so $D\mathcal{J}$ is represented by a symmetric matrix.
		For any matrix $A$, the matrix $A^T\Sigma+ \Sigma A$ represents a tangent vector at the point $\Sigma$ for the submanifold $\mathcal{S}_l\subset Sym^2 V^*$, hence  the flow stays within $\mathcal{S}_l$.
	\end{proof}

	\subsection{Eigenvalue evolution and long time existence}\label{sect:longtimeexistence}

	We shall prove that the Riemannian gradient flow exists for long time. This essentially follows  from the bound on the evolution of the eigenvalues.

	\subsubsection{Diagonal picture}\label{sect:Diagonalpicture}

	In an orthonormal basis of $|\cdot|_V^2$, the $\Sigma(t)$ is represented by symmetric matrices, and the gradient flow reads
	\[
	\frac{d}{dt}\Sigma(t)=-\frac{1}{2}\Sigma D\mathcal{J}(\Sigma;\lambda)- \frac{1}{2}D\mathcal{J}(\Sigma;\lambda) \Sigma.
	\]
	At any given time $t$ as long as the flow exists, we can diagonalize the matrix $\Sigma(t)=\sum \lambda_i(t) e_i^*(t)\otimes e_i^*(t)$, with $\lambda_1\leq \ldots\leq \lambda_d$, and $e_1,\ldots e_d$ here denote the corresponding eigenbasis.

	By the first variation formula in Theorem \ref{thm:firstvariation},
	\[
	D_\Sigma \mathcal{J}(\Sigma, \lambda)=  - \frac{1}{2\lambda}   \E[  r_\Sigma(X,Y) \overline{ r_\Sigma(X',Y') }    K_{\gamma}'( |X-X'|_\Sigma^2) (X-X')\otimes (X-X')].
	\]
	Thus the matrix component $(\frac{d}{dt}\Sigma(t))_{ij}$ is
	\[
	\frac{\lambda_i+\lambda_j}{4\lambda}   \E[  r_\Sigma(X,Y) \overline{  r_\Sigma(X',Y') }    K_{\gamma}'( |X-X'|_\Sigma^2) (X-X')_i(X-X')_j].
	\]
	In particular, the diagonal entries 
	\[
	\frac{d}{dt}\Sigma(t)_{ii} = \frac{\lambda_i}{2\lambda}   \E[  r_\Sigma(X,Y)  \overline{r_\Sigma(X',Y')  }   K_{\gamma}'( |X-X'|_\Sigma^2) |(X-X')_i|^2].
	\]

	\begin{lem}\label{lem:eigenvalueevolution}
		Suppose all the eigenvalues are distinct. Then the eigenvalues and eigenvectors depend smoothly on $t$, and
		\[
		\frac{d}{dt} \lambda_i(t)= \frac{\lambda_i}{2\lambda}   \E[  r_\Sigma(X,Y)  \overline{r_\Sigma(X',Y') }    K_{\gamma}'( |X-X'|_\Sigma^2) |(X-X')_i|^2],
		\]
		and the evolution of the eigenvector is given by
		\[
		\begin{split}
			&  \frac{d}{dt} e_i(t)= \sum_{j\neq i} \frac{ 1 }{\lambda_j-\lambda_i  } (\frac{d \Sigma(t)}{dt})_{ij} e_j
			\\
			= & \sum_{j\neq i} \frac{\lambda_i+\lambda_j}{4\lambda(\lambda_j-\lambda_i)}   \E[  r_\Sigma(X,Y)  \overline{r_\Sigma(X',Y')  }   K_{\gamma}'( |X-X'|_\Sigma^2) (X-X')_i(X-X')_j] e_j.
		\end{split}
		\]

	\end{lem}

	\subsubsection{Long time existence}

	\begin{lem}
		(Matrix maximum principle)
		Suppose $\tilde{\Sigma}(t)$ is a smooth one-parameter family of positive definite inner products, and $\beta$ is a constant, such that
		\[
		\tilde{\Sigma}(0) \geq c|\cdot|_V^2 , \text{ resp. } \tilde{\Sigma}(0) \leq c |\cdot|_V^2 ,
		\]
		and at any time $t$,  for any lowest eigenvector $w$ (resp. highest eigenvector) of the self adjoint operator defined by $\tilde{\Sigma}$ with respect to $|\cdot|_V^2$, we have
		\[
		\frac{d}{dt} \tilde{\Sigma}(t)(w, w) \geq \beta  \tilde{\Sigma}(t) (w, w), \text{ resp. } \frac{d}{dt} \tilde{\Sigma}(t) (w, w) \leq \beta  \tilde{\Sigma}(t)  (w,w).
		\]
		Then for any $t\geq 0$, we have
		$
		\tilde{\Sigma} (t) \geq e^{\beta t} c|\cdot|_V^2 ,
		$
		(resp. $
		\tilde{\Sigma} (t) \leq e^{\beta t}c|\cdot|_V^2  . )
		$

	\end{lem}

	\begin{proof}
		By replacing $\tilde{\Sigma}(t)$ with $\tilde{\Sigma}(t) e^{-\beta t}$, we can reduce to proving the case of $\beta=0$. By replacing $\tilde{\Sigma}$ with the dual inner product $\tilde{\Sigma}^{-1}$, we only need to prove the `resp.' statement. By a standard limiting argument in the proof of strong maximum principles, we can assume the stronger condition
		\[
		\frac{d}{dt} \tilde{\Sigma}(t)(w, w) < 0,
		\]
		and prove $
		\tilde{\Sigma} (t) \leq c I$,
		for all $t\geq 0$.

		Suppose the contrary, then at some time $t$, the inner product $\tilde{\Sigma}(t)$ crosses the boundary of the convex subset $\{ \tilde{\Sigma} \leq c |\cdot|_V^2 \}$. At the boundary point $\tilde{\Sigma}(t)$ of this convex set, the outward normals to the supporting hyperplanes form a convex cone, generated by $w\otimes w$ where $w$ ranges among all the eigenvectors of $\tilde{\Sigma}(t)$ with the highest eigenvalue $c$. By assumption $\tilde{\Sigma}$ crosses some supporting hyperplane of the convex set, so there is some convex combination 
		\[
		\sum_i a_i w_i\otimes w_i,\quad a_i\geq 0,\quad \sum a_i=1, \quad \text{$w_i$ are highest eigenvectors},
		\]
		such that 
		\[
		\langle \frac{d}{dt}\tilde{\Sigma}(t) ,  \sum_i a_i w_i\otimes w_i \rangle \geq 0. 
		\]
		This contradicts the hypothesis. 
	\end{proof}

	We  rule out the divergence to infinity or degeneration in finite time, by an upper bound and a lower bound on $\Sigma(t)$.

	\begin{prop}
		\label{prop:eigenvalueboundevolution}
		In the setting of Theorem \ref{thm:longtimeexistence},
		let
		\[
		\beta_0= \frac{1}{\lambda} \E[|Y|^2 ] \sup_{|w|_V=1}  \E[  |(X, w)_V|^4  ]^{1/2} .
		\]
		Then as long as the flow exists, $\Sigma(t)$  satisfies the bound in time
		\[
		\Sigma (t)\leq  C(\Sigma(0)) |\cdot|_V^2 e^{C(K)\beta_0 t}.
		\]
		where $C(K)$ depends only on the kernel function, and $C(\Sigma(0))$ depends only on  the norm bound of $\Sigma(0)$. 
	\end{prop}

	\begin{proof}
		Let $w$ be a unit eigenvector for the self-adjoint operator defined by $\Sigma(t)$ with respect to the fixed inner product $|\cdot|_V^2$. Then as in Section \ref{sect:Diagonalpicture}, the first variation formula and the definition of the gradient flow imply that
		\[
		\begin{split}
			& \frac{d \Sigma(t)}{dt}  (w,w)
			= 
			\frac{\Sigma(t)(w,w)}{2\lambda}   \E[  r_\Sigma(X,Y)  \overline{r_\Sigma(X',Y') }    K_{\gamma}'( |X-X'|_\Sigma^2)  |(X-X', w)_V|^2 ].
		\end{split}
		\]
		We comment that by treating $w$ as a constant vector in time,  the following argument does not rely on the assumption that the eigenvalues are distinct.

		By assumption the function $K_\gamma'$ is uniformly bounded. Thus by the H\"older inequality, we deduce
		\[
		\begin{split}
			|\frac{d\log \Sigma(t)}{dt}  (w,w)|
			\leq & \frac{C(K)}{\lambda} \E[|r_\Sigma(X,Y)\overline{r_\Sigma(X',Y')}|^2 ]^{1/2} \E[  |(X-X', w)_V|^4  ]^{1/2}
			\\
			= &  \frac{C(K)}{\lambda} \E[|r_\Sigma(X,Y)|^2 ]  \E[   |(X-X', w)_V|^4   ]^{1/2} 
			\\
			\leq &  \frac{C(K)}{ \lambda} \E[|Y-f_U(UX)|^2 ]  \E[  |(X, w)_V|^4 ]^{1/2}
			\\
			\leq &  \frac{C(K)}{\lambda} \E[|Y|^2 ]  \E[  |(X, w)_V|^4 ]^{1/2}.
		\end{split}
		\]
		Here the second line uses the independence of $(X,Y)$ and $(X',Y')$, the third line uses the Minkowski inequality
		\[
		\E[  |X-X'|^4  ]^{1/4} \leq \E[  |X|^4  ]^{1/4} + \E[  |X'|^4  ]^{1/4} =2 \E[  |X|^4  ]^{1/4} ,
		\]
		and the fourth line uses the trivial upper bound \cite[Lemma 2.4]{RuanLi}
		\[
		\lambda \norm{f_U}_\mathcal{H}^2+ \E[ |Y-f_U(UX)|^2] \leq \E[|Y|^2].
		\]
		Thus 
		\[
		|\frac{d\Sigma(t)}{dt} (w,w)| \leq C(K) \beta_0\Sigma(w,w). 
		\]
		so the matrix maximum principle implies the result. 
	\end{proof}

	\begin{proof}
		(Long time existence Theorem \ref{thm:longtimeexistence}). Suppose the maximal time of existence is $T<+\infty$. Then for any $0<t<T$, the flow remains inside a bounded region in $Sym^2_{\geq 0}$ depending only on $T$. Thus the flow can be further extended, contradiction. 
	\end{proof}

	\subsection{Comparison of flows}

	\subsubsection{Riemannian gradient flow vs. Euclidean flow for $U\in \End(V)$}

	\begin{prop}\label{prop:Euclideanflow}
		Under the map
		\[
		\pi: \End(V) \to Sym^2_{\geq 0}, \quad U\mapsto \Sigma= U^T U,
		\]
		the flow on $\End{V}$ defined by  
		\begin{equation}\label{eqn:Uflow}
			\frac{dU_a^i}{dt}= - \frac{1}{2} C_{ac} U_d^i (D\mathcal{J})^{cd}
		\end{equation}
		lifts the flow (\ref{eqn:gradientflow}) on $Sym^2_{\geq 0}$. Moreover, the flow (\ref{eqn:Uflow}) is equivalent to the \emph{Euclidean gradient flow} (\ref{eqn:Euclideanflow}) for the functional $J(U,\lambda)$ on $\End{V}$.
	\end{prop}

	\begin{proof}
		We compute in the index notation, to make manifest the role of the inner product $|\cdot|^2_V= C_{ij}dx^idx^j$; the interested reader may work out the matrix notation proof.

		Since $\Sigma_{ab}= U_a^i C_{ij} U_b^j$, we have $\frac{d}{dt}\Sigma_{ab}= \frac{d U_a^i}{dt} C_{ij} U_b^j+ U^i_a C_{ij} \frac{dU_b^j}{dt}$. We can rewrite the Riemannian gradient flow (\ref{eqn:gradientflow2}) in terms of $U$:
		\[
		\frac{d U_a^i}{dt} C_{ij} U_b^j+ U^i_a C_{ij} \frac{dU_b^j}{dt}= - \frac{1}{2} (U_a^i C_{ij} U_c^j C_{bd} (D \mathcal{J})^{cd}   + U_b^j U_d^i C_{ij} C_{ac}  (D \mathcal{J})^{cd}    ).
		\]
		This is implied by the flow (\ref{eqn:Uflow}) on $U$. This shows that the flow (\ref{eqn:Uflow}) projects down to the Riemannian gradient flow (\ref{eqn:gradientflow2}) under $\pi: \End(V)\to Sym^2_{\geq 0}$.

		Furthermore,  we compute the differential of $J(U,\lambda)= \mathcal{J}(\Sigma)$ by the Leibniz rule and the chain rule:
		\[
		\begin{split}
			DJ= \pi^* D\mathcal{J}= & (D\mathcal{J})^{cd} \pi^* d\Sigma_{cd} = (D\mathcal{J})^{cd} d( U_c^i C_{ij} U_d^j  )
			\\
			= &   (D\mathcal{J})^{cd} C_{ij} U_d^j  dU_c^i  + (D\mathcal{J})^{cd} U_c^i C_{ij} dU_d^j 
			\\
			= &  2(D\mathcal{J})^{cd} C_{ij} U_d^j  dU_c^i  .
		\end{split}
		\]
		The last line here uses the symmetry of $(D\mathcal{J})^{cd}$ and $C_{ij}$. But $DJ= (DJ)_i^c dU_c^i$, so
		\[
		(DJ)_i^c = 2(D\mathcal{J})^{cd} C_{ij} U_d^j  ,\quad    C^{ij}(DJ)^c_j= 2(D\mathcal{J})^{cd} U_d^i. 
		\]
		Thus the flow (\ref{eqn:Uflow}) is equivalent to
		\[
		\frac{dU_a^i}{dt}= - \frac{1}{2} C_{ac} U_d^i (D\mathcal{J})^{cd}= - \frac{1}{4} C_{ac} C^{ij} (DJ)^c_j
		\]
		Up to a constant factor, this is just the Euclidean gradient flow for $J(U,\lambda)$ on the vector space $\End(V)$ with the Euclidean metric induced by $|\cdot|_V^2$.
	\end{proof}

	\begin{enumerate}
		\item  	Even though the Euclidean gradient flow on $\End(V)$ lifts the Riemannian gradient flow on $Sym^2_{\geq 0}$, the Euclidean metric on $\End(V)$ is very far from the pullback of the Riemannian metric on $Sym^2_+$, which vanishes along the fibre directions of $\pi: \End(V)\to Sym^2_{\geq 0}$.

		\item Euclidean gradient flow on $\End(U)$  for a generic objective function fails to preserve the rank of $U$. However, the Euclidean flow for  the functional $J(U,\lambda)$ does preserve the rank, since  the (extended) Riemannian gradient flow preserves the boundary strata $\mathcal{S}_l\subset Sym^2_{\geq 0}$.

		\item The lift of the Riemannian gradient flow to $\End(V)$ is highly non-unique, since we can twist the flow by the action of the orthogonal group $O(V)$ in a time-dependent manner. This $O(V)$  freedom will eventually be exploited in the proof of Theorem \ref{thm:montonicity}, to put $U$ into an advantageous upper triangular form.
	\end{enumerate}

	\subsubsection{Riemannian gradient flow vs. Euclidean flow for $\Sigma\in Sym^2_+$}

	We compare our Riemannian gradient flow, with a natural  Euclidean gradient flow which one may come across in the first attempt.

	A canonical choice of the metric on $Sym^2_+$ is the Euclidean metric, given on the cotangent space of $Sym^2_+$ in terms of the index notation by
	\[
	|A|_{\mathfrak{g}}^2= A^{ac} A^{bd} C_{ab} C_{cd},
	\]
	In an orthonormal basis of $|\cdot|_V^2$, this Euclidean gradient flow is
	\[
	\frac{d}{dt}\Sigma(t)= - D\mathcal{J}(\Sigma;\lambda),
	\]
	while our Riemannian gradient flow is
	\[
	\frac{d}{dt}\Sigma(t)= -\frac{1}{2}\Sigma D\mathcal{J}(\Sigma;\lambda)-  \frac{1}{2}D\mathcal{J}(\Sigma;\lambda) \Sigma. 
	\]
	Both flows respect the rotational symmetry.

	\begin{enumerate}
		\item 
		When it comes to eigenvalue evolution, the difference between these two choices in the gradient flow, is the extra factor $\lambda_i$ in Lemma \ref{lem:eigenvalueevolution}. This factor slows down the evolution when the eigenvalue is small, and accelerates the evolution when the eigenvalue is large. 
		As for the eigenvectors, in the simple case that $\lambda_i\ll \lambda_j$, then the factor $\frac{\lambda_i+\lambda_j}{\lambda_j-\lambda_i}$ in Lemma \ref{lem:eigenvalueevolution} is comparable to one, while the Euclidean gradient descent would produce a factor $ \frac{1}{\lambda_j-\lambda_i} $, making the rotation speed of the eigenvector more dependent on the scale of the eigenvalues.

		\item  In the absence of this damping factor, in the Euclidean gradient flow the smallest eigenvalue can drop to zero in finite time, so \emph{the long time existence cannot hold for the Euclidean gradient flow.}

		\item The Riemannian gradient flow naturally preserves the positive cone $Sym^2_{\geq 0}$ together with the rank stratification on its boundary. The Euclidean gradient flow on $Sym^2_+$ has no such good property.

	\end{enumerate}

	The upshot is that the Riemannian gradient flow is much better behaved.

	\subsection{Time evolution formula}


	Suppose we are given a fixed orthogonal decomposition
	\[
	V=V_1\oplus V_2
	\]
	with respect to the inner product $|\cdot|_V^2$. Let
	$
	\mathcal{U}=\{ U\in \End(V): U(V_1)\subset V_1  \}.
	$
	We have a map $\mathcal{U}\to Sym^2_{\geq 0}$ by $U\mapsto \Sigma=U^T U$, and the differential of this map at a given point $U$ is 
	\begin{equation}\label{eqn:differentialdotU}
		\delta{U}\mapsto U^T \delta{U}+ \delta{U}^T U,\quad \delta{U}\in T_U\mathcal{U} .
	\end{equation}
	We denote 
	\[
	I_{V_1}= \sum_1^{d_{V_1}} e_i\otimes e_i \in Sym^2 V\simeq T_\Sigma^* Sym^2_{\geq 0} ,
	\]
	where $\{ e_i\}$ is any orthonormal basis of $V_1$ with respect to the inner product $|\cdot|_V^2$, and let 
	\[
	\dot{U}= \frac{1}{2} \sum_{i,j=1}^{d_{V_1}} U^i_j \partial_{U^i_j} \in T_U \mathcal{U}
	\]
	denote half of the Euler vector field\footnote{On any finite dimensional vector space $V$, the Euler vector field is $\sum_i x^i\partial_{x^i}$ in any linear coordinate system $x^i$.} on $\End(V_1)$, viewed as an element of 
	\[
	T_U\mathcal{U}\simeq \{ A \in \End{V}: A(V_1)\subset V_1\} \simeq \End(V_1)\oplus \Hom(V_2, V_1)\oplus \End(V_2)
	\]
	using the decomposition $V=V_1\oplus V_2$. The $\Hom(V_2, V_1)\oplus \End(V_2)$ component of $\dot{U}$ is zero.

	\begin{lem}
		Suppose $\Sigma\in Sym^2_+$.
		Under the metric $\mathfrak{g}$, the element of $T_\Sigma Sym^2_+$ metric dual to $I_{V_1}$ is given by the image under (\ref{eqn:differentialdotU}) of the tangent vector $\dot{U}\in T_U\mathcal{U}$. More concretely, the following two equivalent characterizations hold:
		\begin{equation}\label{eqn:IV1dual1}
			\langle U^T \dot{U}+ \dot{U}^T U, A\rangle = (A, I_{V_1})_{\mathfrak{g}}, \quad \forall A\in T_\Sigma^* Sym^2_+\simeq Sym^2 V,
		\end{equation}
		\begin{equation}\label{eqn:IV1dual2}
			( U^T \dot{U}+ \dot{U}^T U, \tilde{A} )_{ \mathfrak{g} } = \langle \tilde{A}, I_{V_1}\rangle , \quad \forall \tilde{A}\in T_\Sigma Sym^2_+\simeq Sym^2 V^*,
		\end{equation}
		where the angle bracket denotes the natural bilinear pairing of the dual spaces $Sym^2 V$ and $Sym^2 V^*$, while $(, )_{\mathfrak{g}}$ denotes the metric inner product on the cotangent/tangent spaces of $Sym^2_+$ respectively.
	\end{lem}

	\begin{proof}
		(Matrix notation proof)
		We compute in an orthonormal basis $\{e_i\}$ of the inner product $|\cdot|^2$, such that $e_1,\ldots e_{d_1}$ give an orthonormal basis of $V_1$, and $e_{d_1+1},\ldots e_d$ is an orthonormal basis of $V_2$. In this basis, the Riemannian metric $\mathfrak{g}$ on $T_\Sigma^* Sym^2_+$ defined in (\ref{eqn:Riemannianmetric}) can be written in matrix notation as 
		\[
		(A,A')_{\mathfrak{g}} = \frac{1}{2}\Tr (A\Sigma A')+ \frac{1}{2}\Tr (A'\Sigma A) ,
		\]
		where $A, A'$ are represented by symmetric matrices. 
		The element $I_{V_1}$ is represented by the diagonal matrix $\text{diag}(1,\ldots 1, 0,\ldots 0)$. The metric dual element of $I_{V_1}$ is by definition the unique element $B\in T_\Sigma Sym^2_+\simeq Sym^2 V^*$, such that for any $A\in T_\Sigma^* Sym^2_+$,
		\[
		\Tr (AB)= (A, I_{V_1})_{\mathfrak{g}}= \frac{1}{2} \Tr (A\Sigma I_{V_1}) + \frac{1}{2}  \Tr (I_{V_1}\Sigma A),
		\]
		so $B$ is represented by the matrix
		\[
		B= \frac{1}{2} (\Sigma I_{V_1}+ I_{V_1}\Sigma) = \frac{1}{2} (U^T U I_{V_1}+ I_{V_1} U^T U)=   (U^T \dot{U} + \dot{U}^T U).
		\]
		This is the image of the element $\dot{U}$ as required.
	\end{proof}

	\begin{proof}
		(Tensor notation proof) In any linear coordinate system $\{x^i\}$ on $V$, 
		the inner product $|\cdot|_V^2$ is $C_{ij}dx^idx^j$. The invertible linear map $U$ is $(U^i_j)$ in coordinates, and the relation $\Sigma=U^T U$ reads
		$
		\Sigma_{ij}= U_i^k C_{kl} U_j^l.
		$
		Here one needs to take care that the adjoint $U^T$ of $U$ involves the inner product $C_{ij}dx^idx^j$. Likewise, the image of $\dot{U}$ under the differential is
		\[
		U^i_c C_{ij} \dot{U}^j_d +  U^i_d C_{ij} \dot{U}^j_c.
		\]
		Using the definition (\ref{eqn:Riemannianmetric}) of the metric $\mathfrak{g}$, we can compute the dual element of $I_{V_1}$ as
		\[
		\frac{1}{2} (\Sigma_{ac} C_{bd} I_{V_1}^{ ab } + \Sigma_{bd} C_{ac} I_{V_1}^{ab})= \frac{1}{2} (U_a^i C_{ij} U_c^j C_{bd} I_{V_1}^{ ab } + U_b^i C_{ij} U_d^j C_{ac} I_{V_1}^{ab}) 
		\]
		This is the image of the tangent vector
		$
		(\dot{U})^i_j=  \frac{1}{2} U^i_a I_{V_1}^{ab} C_{bj}.
		$

		Now we observe that $I_{V_1}^{ab} C_{bj}$ defines the orthogonal projection operator $P_{V_1}$ to $V_1$ with respect to $|\cdot|_V^2$, so $\frac{1}{2} U^i_a I_{V_1}^{ab} C_{bj}$ defines the endomorphism $\frac{1}{2} U P_{V_1}$. By definition $U$ preserves the subspace $V_1$, while $P_{V_1}$ is identity on $V_1$ and zero on $V_2$. Thus $\dot{U}$ agrees with half of the Euler vector field on $\End(V_1)$. 
	\end{proof}

	We can define the partial trace of any $\Sigma\in Sym^2_{\geq 0}$ on the fixed subspace $V_1$,
	\[
	\Tr_{V_1} \Sigma= \sum_1^{d_{V_1}} \Sigma(e_i, e_i)= \langle \Sigma, I_{V_1}\rangle.
	\]
	where $\{ e_i\}$ is an orthonormal basis of $V_1$ with respect to the fixed inner product $|\cdot|_V^2$. Under the gradient flow (\ref{eqn:gradientflow}), this quantity evolves in time according to the following elegant formula. 
	\begin{prop}\label{lem:timeevolution}
		(Time evolution formula)
		\begin{equation}\label{eqn:timeevolutionformulapartialtrace}
			\frac{d}{dt} \Tr_{V_1} \Sigma(t)= -\langle DJ(U,\lambda), \dot{U} \rangle,
		\end{equation}
		where the RHS is the directional derivative of the functional $J(U,\lambda)$ on $\mathcal{U}$ evaluated at the point $U$ along the tangent vector $\dot{U}$ as above. 
	\end{prop}

	\begin{proof}
		We first assume $\Sigma\in Sym^2_+$. 
		By the definition of the gradient flow, the time derivative is computed in terms of the gradient component,
		\[
		\frac{d}{dt} \Tr_{V_1} \Sigma= \langle \frac{d\Sigma(t) }{dt} , I_{V_1}\rangle= -\langle \nabla^{\mathfrak{g}} \mathcal{J}, I_{V_1}\rangle= -( D\mathcal{J}(\Sigma;\lambda), I_{V_1})_{\mathfrak{g}}.  
		\]
		Now using the characterization (\ref{eqn:IV1dual2}) of the metric dual element, we see
		\[
		\frac{d}{dt} \Tr_{V_1} \Sigma=  -( D\mathcal{J}(\Sigma;\lambda), I_{V_1})_{\mathfrak{g}}= - \langle D\mathcal{J}(\Sigma;\lambda), U^T \dot{U} + \dot{U}^T U\rangle. 
		\]
		Under the map $\mathcal{U}\to Sym^2_+$ defined by $\Sigma= U^T U$, the functional $\mathcal{J}(\Sigma, \lambda)$ pulls back to $J(U,\lambda)$, and the tangent vector $\dot{U}$ push forwards to $U^T \dot{U} + \dot{U}^T U$, so by the chain rule the above expression is equal to the directional derivative $-\langle DJ(U,\lambda), \dot{U} \rangle$.

		By continuity, the same formula holds for $\Sigma$ on the boundary of $Sym^2_{\geq 0}$.
	\end{proof}

	\section{Convergence to stationary points}\label{sect:convergence}

	In this Section we prove Theorem \ref{thm:convergence}.
	We assume that $X$ is a continuous random variable, and work in the setting of Theorem \ref{thm:longtimeexistence}. Thus by \cite[Corollary 6.2]{RuanLi},
	\begin{equation}
		\lim_{ \Sigma\to +\infty} \mathcal{J}(\Sigma;\lambda) = \frac{1}{2} \E[|Y|^2].  
	\end{equation}
	Without loss $\E[|  \E[Y|X]  |^2]>0$, so we have the trivial upper bound for $\Sigma\in Sym^2_+$ \cite[Lemma 2.3]{RuanLi} 
	\[
	\mathcal{J}(\Sigma;\lambda) < \frac{1}{2} \E[|Y|^2].
	\]
	Since $\mathcal{J}(\Sigma)$ is monotone non-increasing under the gradient flow, we know $\Sigma(t)$ stays within a bounded region inside $Sym^2_{\geq 0}$ (See Remark \ref{rmk:partialcompactification}). Moreover by (\ref{eqn:Jdecreases}), 
	\[
	\int_0^\infty |D\mathcal{J} |_{\mathfrak{g}}^2  dt\leq \mathcal{J}(\Sigma(0)) \leq \frac{1}{2}\E[|Y|^2].
	\]
	Working in an orthonormal basis of $|\cdot|_V^2$, we regard $D\mathcal{J}$ as symmetric matrices, and $|D\mathcal{J}|_{\mathfrak{g}}^2= \Tr(D\mathcal{J} \Sigma D\mathcal{J} ) $ implies
	\[
	\int_0^\infty \Tr( D\mathcal{J} \Sigma D\mathcal{J}      ) dt\leq \mathcal{J}(\Sigma(0)) \leq \frac{1}{2}\E[|Y|^2].
	\]
	This holds not just for the flow in $Sym^2_+$, but also for flows on the boundary of $Sym^2_{\geq 0}$, by a continuity argument.


	By compactness, we deduce
	there is a subsequence of time $t_i\to +\infty$, such that 
	$\Sigma(t_i)$ converges to some $\Sigma_0\in Sym^2_{\geq 0}$, and 
	$ \Tr( D\mathcal{J} \Sigma D\mathcal{J}      ) (t_i) \to 0.$
	However, by the continuity of the tensor field $D\mathcal{J}$, 
	\[
	\Tr( D\mathcal{J} \Sigma D\mathcal{J}      ) (t_i) \to  \Tr( D\mathcal{J}(\Sigma_0;\lambda) \Sigma_0 D\mathcal{J} (\Sigma_0;\lambda)     )=0.
	\]

	By Gram-Schmidt, we can write $\Sigma_0=U_0^TU_0$ for some $U_0\in \End(V)$. Thus
	\[
	\Tr (    (U_0 D\mathcal{J}(\Sigma_0;\lambda))^T (D\mathcal{J} (\Sigma_0;\lambda)U_0)       ) = \Tr( D\mathcal{J}(\Sigma_0;\lambda) \Sigma_0 D\mathcal{J} (\Sigma_0;\lambda)     )=0.
	\]
	hence $U_0 D\mathcal{J}(\Sigma_0;\lambda)=0$. This shows $D\mathcal{J} (\Sigma_0;\lambda) \Sigma_0=\Sigma_0 D\mathcal{J}(\Sigma_0;\lambda)=0$, hence for any matrix $A$, 
	\[
	\Tr(D\mathcal{J} (\Sigma_0;\lambda) (A \Sigma_0+ \Sigma_0 A^T)     ) =0.
	\]
	But the matrices $A \Sigma_0+ \Sigma_0 A^T$ span the tangent space of $\mathcal{S}_l$ at $\Sigma_0$, hence $D\mathcal{J}(\Sigma_0;\lambda)$ vanishes on the tangent space of 
	$\mathcal{S}_l$, namely $\Sigma_0$ is a stationary point.

	\begin{rmk}
		In applications, if we are looking for stationary points with low rank $l\leq l_0\ll d$, then we can initialize the flow in the low rank strata $\bigcup_{l\leq l_0} \mathcal{S}_l$.
		Under the conditions of Section \ref{sect:extensionflow}, 
		this subset is preserved under the flow, and is closed inside $Sym^2 V^*$, hence contains the subsequential limits of $\Sigma(t_i)$ as $t_i\to +\infty$. The low rank condition may lead to significant improvement in computational efficiency.
	\end{rmk}

	\section{Gradient flow dynamics: de-noising}

	\subsection{Gaussian noise: single variable case}\label{sect:Gaussiansingle}

	We start with the following simple situation. Consider a vector space $V\simeq \R^d$ with a given direct sum decomposition into $\R\oplus V'$, where $\R$ is a one-dimensional subspace. As usual $X$ is a random variable valued in $V$, so $X$ has two components $(X_1, X')\in \R\oplus V'$. We assume that $X_1$ is a \emph{Gaussian variable} $X_1\sim N(0, \alpha^2) $  independent of $X'$, and the random variable $\E[Y|X]=\E[Y|X']$ depends only on $X'$, and $\E[|Y|^2]<+\infty$. The distribution law of $X'$ is arbitrary. 
	In this section, the kernel function is not required to be rotationally invariant.

	Let $U\in \End(V)$ be block triangular, namely
	\[
	U= \begin{bmatrix}
		U_{11}, & U_{12}\\
		0, & U_{22}
	\end{bmatrix} \in \begin{bmatrix}
		\R, & \Hom(V', \R) \\
		0, & \End(V')
	\end{bmatrix}.
	\]
	We denote $H_s$ as the \emph{convolution type operation} 
	\[
	H_sf(x)= \frac{1}{\alpha \sqrt{2\pi s}  } \int_{ \R } f(x_1+ U_{11} y, x' ) \exp( -\frac{|y|^2}{2\alpha^2 s}) dy,\quad s > 0.
	\]
	This has a more probablistic interpretation: take a Gaussian random variable $Z\sim N(0,\alpha^2s)$, then
	\[
	H_s f(x)= \E_Z[   f(x+ (U_{11} Z,0) )  ]
	\]
	Note that $\{H_s\}_{s > 0}$ forms a Markov semigroup.

	\begin{lem}\label{lem:convolution1}
		Let $f$ be any bounded function on $V$, then
		\[
		\E [| Y-(H_s f)( (\sqrt{1-s} U_{11}X_1 +U_{12} X', U_{22} X'))|^2] \leq \E[|Y-f(UX)|^2].
		\]

	\end{lem}

	\begin{proof}
		We can replace $Y$ by $\E[Y|X]$ on both sides, by subtracting the same quantity $\E[  {\rm Var}(Y|X)  ]$ on both sides. By assumption $\E[Y|X]=\E[Y|X']$ depends only on $X'$, so without loss $Y$ is a function of $X'$.

		The Gaussian variable $X_1$ is equivalent in law to the sum of two independent Gaussian variables 
		$Z_1, Z_2$ of variance $\alpha^2 s$ and $\alpha^2(1-s)$. Since $X_1, Z_1, Z_2$ are independent from $X'$, we have
		\[
		\E[|Y-f(UX)|^2] = \E[  |Y-f(U(X_1, X'))|^2  ]=  \E[|Y-f(U(Z_1+Z_2, X'))|^2] .
		\]
		Conditional on the value of $Z_2, X'$, the probablistic interpretation of $H_sf$ gives
		\begin{equation*}
			\begin{split}
				&     | Y-(H_s f)(U(Z_2, X'))|^2
				\\
				=&| Y-\E_{Z_1}[ f(U(Z_1+Z_2, X'))|  Z_2, X']  |^2
				\\
				=|& \E[ Y- f(U(Z_1+Z_2, X'))|  Z_2, X']  |^2
				\\
				\leq &  \E[      |Y-f(U(Z_1+Z_2), X')   |^2        |Z_2, X'  ].
			\end{split}
		\end{equation*}
		The third line uses that $Y$ 
		is a function of $X'$, and the fourth line uses Cauchy-Schwarz. 
		By integrating further in $Z_2, X'$, we obtain
		\begin{equation}
			\E [| Y-(H_s f)(U(Z_2, X'))|^2]\leq   \E[|Y-f(U(Z_1+Z_2, X'))|^2] =   \E[|Y-f(UX)|^2].
		\end{equation}

		On the other hand, by the change of variable formula $Z_2$ is equivalent in law to the Gaussian variable $\sqrt{1-s} X_1$, and since $Z_2, X_1$ are both independent from $Y, X'$, we deduce
		\[
		\begin{split}
			\E [| Y-(H_s f)(U(Z_2, X'))|^2] 
			= \E [| Y-(H_s f)(U( (\sqrt{1-s} X_1, X'))|^2] .
		\end{split}
		\]
		Combining the above,
		\[
		\E [| Y-(H_s f)(U( (\sqrt{1-s} X_1, X'))|^2] \leq \E[|Y-f(UX)|^2]
		\]
		as required.
	\end{proof}

	\begin{lem}\label{lem:convolution2}
		Let $f\in \mathcal{H}$. 
		The Fourier transform of the function $H_sf $ is 
		\[
		\widehat{ H_s f}(\omega)=  e^{-2\pi^2\alpha^2 U_{11}^2\omega_1^2 s} \hat{f}(\omega).
		\]
	\end{lem}

	\begin{proof}
		Since $f\in \mathcal{H}$, we know $\hat{f}\in L^1$.
		We compute the Fourier transform 
		\[
		\begin{split}
			& \int_{V} H_sf(x)e^{-2\pi i \langle \omega, x\rangle} dx
			\\
			& = \frac{1}{\alpha \sqrt{2\pi s}  } \int_{ \R }\int_{V} f(x_1+ U_{11} y, x' ) \exp( -\frac{|y|^2}{2\alpha^2 s}) e^{-2\pi i \langle \omega, x\rangle} dx dy
			\\
			& =
			\frac{1}{\alpha \sqrt{2\pi s}  } \hat{f}(\omega) \int_{ \R } \exp( -\frac{|y|^2}{2\alpha^2 s}) e^{2\pi i \langle \omega_1, U_{11}y\rangle}  dy
			\\
			& = e^{-2\pi^2\alpha^2U_{11}^2 \omega_1^2 s} \hat{f}(\omega)
		\end{split}
		\]
		as  required.
	\end{proof}

	We now consider a family of matrices 
	\[
	U_s= \begin{bmatrix}
		\sqrt{1-s}U_{11}, & U_{12}\\
		0, & U_{22}
	\end{bmatrix} \in \begin{bmatrix}
		\R, & \Hom(V', \R) \\
		0, & \End(V')
	\end{bmatrix}.
	\]
	We will be interested in the minimizer $f_U$ of the loss function
	\[
	I(f, U, \lambda)= \frac{1}{2} \E[|Y-f(UX)|^2] + \frac{\lambda}{2} \norm{f}_{\mathcal{H}}^2,\quad J(U,\lambda)=\min_f I(f, U, \lambda).
	\]

	\begin{cor}\label{cor:convolution}
		We have
		\[
		J(U_s,\lambda)\leq  J(U,\lambda)-  \frac{\lambda}{2}\int_{V^*} (1-e^{-4\pi^2\alpha^2U_{11}^2\omega_1^2 s}) \frac{ |\hat{f}_U|^2}{k_V} d\omega.
		\]
	\end{cor}
	
	\begin{proof}
		Given $U$, the minimizer $f=f_U$ is in $\mathcal{H}$, so in particular $\hat{f}\in L^1$ and $f$ is bounded. We observe the norm contraction property
		\[
		\norm{ H_s f}_{\mathcal{H}}^2= \int_{V^*} e^{-4\pi^2\alpha^2U_{11}^2\omega_1^2 s} \frac{ |\hat{f}(\omega)|^2}{k_V} d\omega\leq  \int_{V^*}  \frac{ |\hat{f}|^2}{k_V} d\omega= \norm{f}_{\mathcal{H}}^2<+\infty,
		\]
		so $H_s f\in \mathcal{H}$. We use $H_sf$ as a test function. By Lemma \ref{lem:convolution1} and \ref{lem:convolution2},
		\[
		\begin{split}
			& I(H_s f, U_s, \lambda)=
			\frac{1}{2}\E [| Y-(H_s f)( U_s X)|^2]+ \frac{ \lambda  }{2} \norm{H_sf}_{\mathcal{H}}^2 
			\\
			& \leq \frac{1}{2}\E[|Y-f(UX)|^2]+ \frac{  \lambda }{2}\int_{V^*} e^{-4\pi^2\alpha^2U_{11}^2\omega_1^2 s} \frac{ |\hat{f}|^2}{k_V} d\omega
			\\
			& = J(U,\lambda)-  \frac{\lambda }{2}\int_{V^*} (1-e^{-4\pi^2\alpha^2U_{11}^2\omega_1^2 s}) \frac{ |\hat{f}|^2}{k_V} d\omega
		\end{split}.
		\]
		This implies the Corollary.
	\end{proof}

	\begin{cor}\label{cor:Gaussianonevariable} 
		
		The `partial derivative' of $J(U,\lambda)$ fixing $U_{12}, U_{22}$ satisfies
		\[
		\limsup_{s\to 0^+} \frac{J(U_s, \lambda)-J(U,\lambda)}{s/2}  \leq -4\pi^2\alpha^2 U_{11}^2 \lambda\int_{V^*} \omega_1^2  \frac{ |\hat{f}_U|^2}{k_V} d\omega \leq 0.
		\]
		
	\end{cor}

	\begin{proof}
		By Corollary \ref{cor:convolution},
		\[
		\begin{split}
			&	\limsup_{s\to 0^+} \frac{J(U_s, \lambda)-J(U,\lambda)}{s/2}  \leq \limsup_{s\to 0^+} - \lambda \int_{V^*} \frac{(1-e^{-4\pi^2\alpha^2  U_{11}^2 \omega_1^2 s})}{s} \frac{ |\hat{f}_U|^2}{k_V} d\omega
			\\
			\leq &  -4\pi^2\alpha^2 U_{11}^2\lambda\int_{V^*} \omega_1^2  \frac{ |\hat{f}_U|^2}{k_V} d\omega .
		\end{split}
		\]
		Here the second line follows from Fatou's Lemma. 
	\end{proof}

	\begin{lem}\label{lem:expdecayrate}
		Suppose $\E[|Y|^2] <+\infty$.
		For $U$ inside any given compact subset in $\End(V)$ where $\E[Y|UX]$ is nonzero (or equivalently, $J(U,\lambda)<\frac{1}{2}\E[|Y|^2]$), we have
		\[
		\inf_U\int_{V^*} \omega_1^2  \frac{ |\hat{f}_U|^2}{k_V} d\omega >0.
		\]

	\end{lem}

	\begin{proof}
		It suffices to show that for any convergent sequence $U_i\to U_\infty$ in this compact set, then 
		\[
		\liminf \int_{V^*} \omega_1^2  \frac{ |\hat{f}_{U_i}|^2}{k_V} d\omega  \geq  \int_{V^*}  \omega_1^2  \frac{ |\hat{f}_{U_\infty}|^2}{k_V} d\omega .
		\]
		If the RHS is zero, then $f_{U_\infty}=0$, so $J(U_\infty,\lambda)=\frac{1}{2}\E[|Y|^2]$ saturates the trivial upper bound, or equivalently $\E[Y|U_\infty X  ]=0$ by \cite[Lem. 2.3]{RuanLi}; this contradiction shows the RHS is strictly positive. As a caveat, we allow that $\int_{V^*} \omega_1^2  \frac{ |\hat{f}_U|^2}{k_V} d\omega $ may be infinite.

		By \cite[Corollary 2.9]{RuanLi}, the minimizer function $f_{U_i}\to f_{U_\infty}$ strongly in the $\mathcal{H}$-norm. Recall that $\norm{f}_{\mathcal{H}}^2= \int_{V^*}  \frac{ |\hat{f}_U|^2}{k_V} d\omega$ is a weighted $L^2$-norm on the Fourier transform. Thus for any $R>0$, we have
		\[
		\liminf \int_{ V^*}  \omega_1^2  \frac{ |\hat{f}_{U_i}|^2}{k_V} d\omega\geq 	\lim_{i\to +\infty}	\int_{ |\omega|\leq R}  \omega_1^2  \frac{ |\hat{f}_{U_i}|^2}{k_V} d\omega = \int_{ |\omega|\leq R}  \omega_1^2  \frac{ |\hat{f}_{U_\infty}|^2}{k_V} d\omega
		\]
		Since this holds for all $R>0$, the claim follows.
	\end{proof}

	\subsection{Gaussian noise: multiple variable case}\label{sect:Gaussianmultiple}

	We now generalize the situation to multiple Gaussian variables, as in the setting of Theorem \ref{thm:montonicity}. We suppose $V$ is given a direct sum decomposition into $W\oplus W'$, so the random variable $X$ has two components $(X_W, X_{W'})\in W\oplus W'$. We assume that $X_W$ is a \emph{Gaussian vector} $X_W\sim N(\E[X_W], \Cov(X_W)) $  independent of $X_{W'}$, and the random variable $\E[Y|X]$ is a function of $X_{W'}$ only, and $\E[|Y|^2]<+\infty$. The distribution law of $X_{W'}$ is arbitrary, subject to the requirement that $\E|X'|^2<+\infty$. We choose $|\cdot|_V^2$ to be defined by the inverse covariance matrix. By the independence of $X_W, X_{W'}$,  the decomposition $W\oplus W'$ is \emph{orthogonal}.

	The key advantage of \emph{rotational invariance} of the kernel is that the minimization problem $J(U,\lambda)$ depends only on $\Sigma=U^T U$, so we will use the flexibility to \emph{choose the factorization} $\Sigma=U^TU$ to our advantage.

	\begin{thm}\label{thm:Gaussiannoisemonotonicity}
		(Monotonicity formula in the flow) In the setting of Theorem \ref{thm:longtimeexistence},
		given any unit vector $w\in W$ with respect to $|\cdot|_V^2$, then along the Riemannian gradient flow, we have
		\begin{equation}\label{eqn:monotonicity}
			\frac{d}{dt}|w|_{\Sigma(t)}^2 \leq - 2\pi^2\lambda|w|_{\Sigma(t)}^2 \int_{V^*} \omega_1^2 \frac{ |\hat{f}_U|^2}{k_V} d\omega\leq 0.
		\end{equation}
		In particular, if the initial data $\Sigma(0)$ satisfies $\mathcal{J}(\Sigma(0);\lambda)<\frac{1}{2}\E[|Y|^2]$, and $\Sigma(t)$ stays in a compact region of $Sym^2_{\geq 0}$, then $|w|_{\Sigma(t)}$ decays exponentially in $t$. 
		Consequently, any stationary point $\Sigma\in Sym^2_{\geq 0}$ satisfies $|w|_\Sigma=0$, or else $\E[Y|UX]=0$ (or equivalently $J(U,\lambda)=\frac{1}{2}\E[|Y|^2]$).
		
	\end{thm}

	\begin{proof}
		Given the unit vector $w\in W$, we can choose its orthogonal complement $W_1$ in $W$ with respect to $|\cdot|_V^2$. This induces an orthogonal decomposition
		\[
		V=W\oplus W'= \R w\oplus W_1 \oplus W'= \R w\oplus  (W_1 \oplus W').
		\]
		The random variable $X$ can be decomposed into the $\R$-component $X_1$, which is by definition the coefficient of $w$, and the $(W_1 \oplus W')$-component $X'$. 
		By assumption $X_1$ is independent of $X_{W'}$, and by the covariance matrix construction and the Gaussian property, $X_1$ is also independent of the component of $X'$ in $W_1$. Since the RKHS is translation invariant, we may assume $\E[X]=0$.

		Given any $\Sigma\in Sym^2_{\geq 0}$, by the Gram-Schmidt process we can choose $U$ to be an upper triangular matrix with respect to the direct sum decomposition $\R\oplus  (W_1 \oplus W')$, such that $\Sigma=U^T U$. (This works even if $\Sigma$ is only semi-positive.) The upshot is that we have \emph{reduced the multi-variable case to the single variable case}. It remains to reinterpret the conclusions of Section \ref{sect:Gaussiansingle}.

		By Proposition \ref{lem:timeevolution}, the time derivative
		\[
		\frac{d}{dt}|w|_{\Sigma(t)}^2 =- \langle DJ(U,\lambda), \dot{U} \rangle,
		\]
		which is simply the directional derivative of $J(U,\lambda)$ with respect to the tangent vector $\frac{1}{2}U_{11}\partial_{U_{11}}$. 
		But the LHS expression in Corollary \ref{cor:Gaussianonevariable} is precisely
		\[
		\lim_{s\to 0} \frac{J(U_s, \lambda)-J(U,\lambda)}{s/2} = - \langle DJ(U,\lambda), U_{11}\partial_{U_{11}} \rangle= 2\frac{d}{dt}|w|_{\Sigma(t)}^2, 
		\]
		so Corollary \ref{cor:Gaussianonevariable} implies
		\[
		\frac{d}{dt}|w|_{\Sigma(t)}^2\leq -2\pi^2  U_{11}^2 \Cov(X_1,X_1)\lambda \int_{V^*} \omega_1^2  \frac{ |\hat{f}_U|^2}{k_V} d\omega\leq 0.
		\]

		Finally, since $w$ is a unit vector with respect to the inverse covariance matrix $|\cdot|_V^2$, the factor
		$
		\Cov(X_1, X_1)= \frac{1}{ |w|_V^2 }=1.$ Since $\Sigma = U^T U$, and $U$ is upper triangular, we have  
		\begin{equation*}
			|w|_\Sigma^2 = |Uw|_V^2 = U_{11}^2 |w|_V^2 = U_{11}^2. 
		\end{equation*}
		Thus $U_{11}^2 \text{Cov}(X_1,X_1)= |w|_\Sigma^2$, hence (\ref{eqn:monotonicity}).

		If $\Sigma(t)$ stays in a compact region, then so does the corresponding $U$ satisfying $\Sigma= U^T U$.  By the monotone decreasing of $\mathcal{J}(\Sigma;\lambda)$ in time, the condition $J(U,\lambda)<\frac{1}{2}\E[|Y|^2]$ is preserved by the flow. 
		The exponential decay statement then follows from Lemma \ref{lem:expdecayrate}.
	\end{proof}

	\section*{Postscript: comparison with 2-layer networks}

	\subsection*{Analogy with 2-layer neural networks}

	The kernel learning problem fits into a family of algorithms known as composition models (\ie involving composition of functions), of which the 2-layer neural network is the best known example.

	We recall the setup of the \emph{2-layer neural network} \cite{Bach}.  Let $X$ be a random variable valued in a $d$-dimensional vector space $V$, and let $Y$ be a complex valued random variable with $\E[|Y|^2]<+\infty$. Let $\sigma: \R\to \R$ be a fixed choice of `activation function', such as the ReLU function. Let $N\gg 1$ be a large integer, known as the width of the network. We take the parameters $\underline{\theta}=(a_i, w_i, b_i)_{i=1}^N$, where $a_i\in \C, w_i\in V^*, b_i\in \R$. We can write
	\[
	F: V\to \C,\quad	F(x; \underline{\theta} )= \sum_1^N a_i \sigma( \langle w_i, x\rangle + b_i).
	\]
	The goal is to 
	dynamically update the choice of $a_i, w_i, b_i$, so that the random variable $F(X; \underline{\theta})$ best approximates $Y$, in a manner which is generalizable to unseen data sets. Modulo implementation details, the algorithm is the following:

	\begin{itemize}
		\item Define the \emph{loss function} 
		\[
		L(X; \underline{\theta},\lambda )= \frac{1}{2} | Y- F(X; \underline{\theta})|^2 + \frac{\lambda}{2} \mathcal{R}(\underline{\theta}).
		\]
		The first term here measures how well $F(X;\underline{\theta})$ approximates $Y$, and the second term is designed for \emph{regularization effects}, which ensures good generalization behavior to new data sets. Here $\lambda>0$ is a small parameter, and $\mathcal{R}$ is a function of $\underline{\theta}$, typically chosen to be some $l^2$ or $l^1$ type norm function in the parameters $a_i, w_i, b_i$, such as $\sum |a_i|^2$.

		\item  Upon the \emph{choice of a  metric} in the space of parameters $\underline{\theta}$, one can compute the \emph{partial gradient} of the loss function in the $\underline{\theta}$ variable $\nabla_{\underline{\theta}} L(X;\underline{\theta},\lambda) $.

		\item We sample from the training data set, to compute the empirical expectation $\E_{emp}(  \nabla_{\underline{\theta}} L(X;\underline{\theta})  )$. We then update the parameters $\underline{\theta}$ by the \emph{discretized gradient descent}
		\[
		\underline{\theta}\to \underline{\theta}- \tau \E_{emp}(  \nabla_{\underline{\theta}} L(X;\underline{\theta},\lambda)  )
		\]
		where $\tau>0$ is a small time step parameter. We \emph{iterate} this algorithm until the loss function reaches an approximate \emph{local minimum}, a process known as \emph{training}. Often in practice, one performs the stochastic gradient descent instead, which requires a smaller number of data samples, and is computationally faster.
		
	\end{itemize}

	We can now draw the analogy between our kernel learning problem and the 2-layer neural network. 
	
	\begin{enumerate}
		\item The \emph{loss function} $I(f,U,\lambda)$ is the analogue of $\E[L(X; \underline{\theta},\lambda)]$. Here we are separating the parameters $\underline{\theta}$ into two parts. The \emph{outer layer} coefficient parameters $a_i$ are the analogue of the function $f$, in the sense that the loss function depends \emph{quadratically} on $a_i$ (resp. $f$).
		The \emph{inner layer} parameters $w_i\in V^*$ are analogous to $U\in \End{V}\simeq V^*\otimes V$. The reason that $b_i$ does not have an explicit analogue in our kernel learning problem, is that our RKHS is translation invariant, so a translation parameter can be eliminated by the replacement $f(\cdot)\to f(\cdot-c)$. The dimension for the linear span of $w_i$ is analogous to the rank of  $U$.

		\item   To go further we need to make some simplifications. First, we will ignore the difference between the \emph{expectation}  vs. finite sample \emph{empirical expectation}. Much progress has been made in the statistics literature~\cite{DevroyeGyLugosi, GyorfiKoKrWa}.

		Second, we will ignore the effect of \emph{time discretization}, and treat the discretized gradient descent as a gradient flow. This is actively studied in the optimization literature~\cite{HelmkeMoore, KushnerYin}.

		Third, we will separate the training process in the 2-layer neural network into two steps. For fixed choice of $\theta=(w_i, b_i)$, the training on the outer layer parameters $a_i$  is just linear regression, for which very fast algorithms are available. We will thus first minimize $\E[L(X; \underline{\theta} ,\lambda   )]$ with respect to $a_i$ while fixing $\theta=(w_i, b_i)$, and compute 
		\[
		J_{NN}(\theta, \lambda)= \min_{a_i} \E[  L(X; \underline{\theta}=(a_i, w_i,b_i) ,\lambda   )         ]
		\]
		This $J_{NN}(\theta,\lambda)$ is a function of the inner layer parameters only, and this is the analogue of $J(U,\lambda)$ in the kernel learning problem. Unlike the outer layer, the function $J_{NN}$ (resp. $J(U,\lambda))$ is generally speaking highly nonlinear.

		\item  Now the 2-layer neural network \emph{training dynamics} becomes the adiabatic motion of the parameters $\theta$:
		\begin{equation}\label{eqn:gradflowNN}
			\frac{d\theta }{dt}= - \nabla_\theta J_{NN} (\theta, \lambda).
		\end{equation}
		This is a \emph{gradient flow on the parameter space} for $\theta$, which is $\bigoplus_1^N (V^*\oplus \R)$.
		As an issue not often emphasized in the neural network literature,  to identify the differential $DJ_{NN}$ as a gradient vector field $\nabla J_{NN}$, one has to make a choice of a Riemannian metric on the parameter space, and this choice is part of the design of the gradient descent algorithm \cite{AbsilMaSe08, Boumal20}.

		In practice, one usually choose a \emph{Euclidean metric} on the vector space of parameters. Furthermore, a common technique known as \emph{batch normalization}~\cite{IoffeSzegedy}, morally means that the choice of this metric should be specified by the \emph{inverse covariance matrix} of the random variable $X$, namely $Cov(X)^{-1}=C_{ij}dx^idx^j$. \footnote{In coordinates, the differential of a function is $df=\partial_i fdx^i$, while the gradient vector field is $\nabla f= g^{ij} \partial_i f \frac{\partial}{\partial x^j}$, which requires choosing a metric tensor $g^{ij}$. Since in practice people usually work inside some $\R^N$, the standard Euclidean metric is often taken for granted. Thus the process of normalization on $X$ by applying a linear change of coordinate, can be interpreted as a choice of the Euclidean metric on $V$.}

		Here is the main point of the analogy. The gradient flow (\ref{eqn:gradflowNN}) for the neural network  is analogous to the gradient flow on the vector space $\End(V)$, namely the parameter space of $U$:
		\[
		\frac{dU}{dt}   = -  \nabla_U J(U,  \lambda).
		\]
		Again this gradient requires a choice of the metric. The upshot of this paper is that the  \emph{Euclidean gradient flow} (\ref{eqn:Euclideanflow}) on $U\in \End{V}$ with respect to the functional $J(U,\lambda)$, which corresponds to the \emph{Riemannian gradient flow} (\ref{eqn:Riemannianmetric}) under the map $U\mapsto \Sigma=U^TU$, is the \emph{canonical choice} that satisfies the remarkable properties of noise reduction (Theorem \ref{thm:montonicity}) and  preservation of rank stratification (Theorem \ref{thm:longtimeexistence}). The choice of metric is precisely the Euclidean metric on $\End(V)$ induced by the inverse covariance matrix $C_{ij}dx^idx^j$ on $V$.

		No such results seem to have been proven on the neural network side of the analogy.

		\item  On the negative side of the analogy, we point out that 2-layer neural network is \emph{not} the same algorithm as our kernel learning problem.  The main differences lie in how the functions are represented ($F(X;\underline{\theta})$ vs. $f(UX)$ for some $f\in \mathcal{H}$), and the choice of regularization terms ($\mathcal{R}(\underline{\theta})$ vs. $\norm{f}_{\mathcal{H}}^2$).

	\end{enumerate}

	\subsection*{Some major problems in 2-layer neural networks}

	While the huge literature on 2-layer neural networks have made substantial progress on the topics of approximation theory (e.g., \cite{Cybenko, Barron}, \cite[Chapter 4]{Bach}), the gradient descent algorithm (e.g., \cite{BachChizat}, \cite[Chapter 5]{Bach}), the infinite width limit (e.g., \cite{RotskoffVandenEijnden, MeiMontanariNguyen}) etc, the following central questions are still largely open.

	We refer to questions concerning stationary points of the loss function as \emph{static}, and questions related to the training process as \emph{dynamic}.

	\begin{Question}
		(Landscape)
		Can we describe the landscape of the stationary points of the expected loss function  $\E[ L(X; \underline{\theta}, \lambda)  ]$?
	\end{Question}

	This is a very difficult question, since in general the functional $\E[ L(X; \underline{\theta},\lambda)  ]$ is a highly \emph{nonconvex} function in its parameters $\underline{\theta}$. A notable exception is the infinite width limit, where the loss functional becomes convex \cite{BachChizat}.

	In general one expects that there may be very many stationary points. But not all of them are equally significant: if we encounter \emph{saddle points} or \emph{shallow local minima}, then the gradient descent algorithm will avoid them after small perturbation, so we do not expect them to show up as infinite time limit of the training process~\cite{Pemantle90}. The upshot is that the significant stationary points should be \emph{deep local minima}.

	\begin{Question}
		(Interpretability)  What kind of information do the deep local minima signify? How far can humans interpret this information, and how closely does the algorithm reflect human thinking process?
	\end{Question}

	This question is inherently interdisciplinary, but mathematics is an important perspective.

	\begin{Question}
		(Dynamics)
		What happens to the long time behavior of the training process? Under what conditions can we guarantee convergence to some stationary point? How does the algorithm improve the parameters $\underline{\theta}$, beyond the monotonicity of the loss function?
	\end{Question}

	Apart from the infinite width limit, the known results usually assumes very specific distributions of $X$, and $Y$ is a very specific function of $X$ (See the survey in \cite[Appendix B]{MisiakiewiczMo23}).

	\subsection*{A brief recap of results and insights}

	The above three questions can be asked verbatim for the kernel learning problem. The companion paper \cite{RuanLi} is devoted to the static questions, while this paper makes progress on the dynamical questions. We now make an informal survey of some results therein.

	To address the \textbf{landscape question}, we need to first understand how $J(U,\lambda)$ depends on $U$, or equivalently how $\mathcal{J}(\Sigma;\lambda)$ depends on $\Sigma$. There are two basic ways a sequence of $\Sigma\in Sym^2_+$ can degenerate, and we need to know the \emph{limiting behavior} of $\mathcal{J}(\Sigma;\lambda)$; the mathematical framework is known as the \emph{partial compactification} of $Sym^2_+$.
	\begin{itemize}
		\item The sequence of $\Sigma$ can tend to a boundary point of $Sym^2_{\geq 0}$. In \cite[Section 4]{RuanLi}, we deal with the question of \emph{continuous extension} of the functional $\mathcal{J}(\Sigma;\lambda)$ and its first variation $D\mathcal{J}$. This information is essential for defining the Riemannian gradient flow in this paper.

		\item Or $\Sigma$ can tend to infinity. We make an intriguing observation in \cite[Section 3]{RuanLi} that the limiting behavior of $J(U,\lambda)$ is \emph{sensitive to whether the marginal distribution for $X$ along a linear projection has any discrete structure}.
	\end{itemize}

	We then study the local minima of $\mathcal{J}(\Sigma;\lambda)$, refered to as \emph{vacua} (=energy local minima in physics) in \cite{RuanLi}. In \cite[Section 5]{RuanLi}, we give criterion to test whether a boundary stationary point $\Sigma$ is locally minimizing. From a more applied perspective, local minima $\Sigma$ on the boundary strata of $Sym^2_{\geq 0}$ means that $U$ has \emph{lower rank} than $d=\dim V$, so that $f(UX)$ depends only on a linear projection of $X$, a phenomenon known as \emph{variable selection}(=\emph{feature learning}), and is useful for compressing high dimensional data into a smaller number of degrees of freedoms.

	In \cite[Section 6, 7]{RuanLi}  we develop more tools to analyze what happens to $\mathcal{J}(\Sigma;\lambda)$ when the distribution of $X$ is the superposition of a number of \emph{clusters}, and each cluster has some characteristic \emph{scale parameter}, that measures the local concentration of probability density. In the presence of many scale parameters, or far separated clusters, we exhibit examples where $\mathcal{J}(\Sigma;\lambda)$ have \emph{many deep vacua}, a phenomenon  opposite to what happens for convex functionals. Furthermore, under some additional conditions, one can classify all the deep vacua $\Sigma$ up to uniform equivalence, and show that these vacua precisely \emph{detect the scale parameters of the individual clusters}.

	This goes some way towards the \textbf{interpretability question}, and to explain this we will make some philosophical digression. Humans perceive two kinds of information: \emph{qualitative/conceptual information} is modelled on \emph{discrete random variables}, and \emph{quantitative infomation} is modelled on the \emph{continuous random variables}. In a world of imprecise observations, an underlying discrete structure is manifested in the \emph{concentration} of probability density, as in the case of approximate delta functions. Human learning is largely about the following tasks:
	\begin{enumerate}
		
		\item Understand the \emph{essential degrees of freedom} that most strongly affects the prediction of $Y$. This is the variable selection task.
		
		\item  Extract the qualitative information, which means uncovering an underlying discrete structure in data, that may not be presented to us in an a priori explicit way. The presence of many \emph{clusters} can be viewed a \emph{discrete classification} of the learning problem into several cases.
		
		\item  Understand the \emph{scale parameters}, which determines the domain of validity for the distinct subcases of the discrete classification. 
	\end{enumerate}

	This discussion suggests that the problem of finding the vacua of $\mathcal{J}(\Sigma;\lambda)$ agrees with key aspects of human learning. The main limitation of our kernel learning algorithm, is that we need to assume clustering for the marginal probability density of \emph{linear projections} of $X$ in \cite{RuanLi}, or morally speaking, \emph{the feature variables are linear projections of $X$}. In the real world data, one expects that feature variables are highly nonlinear functions of $X$. This limitation is an inherent flaw for any 2-layer composition model; how to adapt the algorithm to emulate deeper neural networks is a very interesting and open question.

	To find these vacua of $\mathcal{J}(\Sigma;\lambda)$, one needs to address the \textbf{dynamics question}, and this paper makes partial progress towards this. At a qualitative level, Theorem \ref{thm:longtimeexistence}, \ref{thm:convergence} tell us that the algorithm can be run for a long time, without finite time divergence to infinity or losing the postive-semidefiniteness of $\Sigma$, and under mild conditions will yield a \emph{stationary point}. The more surprising aspects are the exponential decay of Gaussian noise (Theorem \ref{thm:Gaussiannoisemonotonicity}), and the preservation of rank stratification (Theorem \ref{thm:longtimeexistence}) on $Sym^2_{\geq 0}$. These results are surprising in the sense that no parametric assumption is made about the feature variables in $X$, or how $Y$ depends on these variables, and in particular $\mathcal{J}(\Sigma;\lambda)$ can be highly non-convex. Our result says that not only the functional $\mathcal{J}(\Sigma;\lambda)$ decreases, which holds by fiat, but \emph{a continuous family of functionals decrease in time simultaneously along the  gradient flow.}

	Finally, back to the elephant in the room: what does this mean for 2-layer neural networks?
	The majority of these results above have no known analogue that is proven for 2-layer neural networks at any degree of generality. There are two attitudes one may take, and we leave this for the reader to decide:
	\begin{enumerate}
		\item The kernel learning problem is an easier toy model of 2-layer neural networks, and one should aim to prove the analogous results in the neural network case.

		\item The kernel learning problem enjoys some very special mathematical properties that have no analogue among similar models, but neural network has much better developed numerical implementation schemes, that could be adapted to solve the kernel learning problem efficiently\footnote{In the kernel learning model $f(UX)$, one can approximate the function $f\in \mathcal{H}$ via a random feature expansion~\cite{RahimiRecht}, expressing it as a linear combination of random nonlinear features. This approximation resembles a single hidden layer in a neural network, thereby enabling neural network architectures as efficient numerical schemes for kernel-learning models.}.

	\end{enumerate}

	\subsection*{Open problems}

	We list some open questions which are important for theoretical or numerical aspects of the kernel learning problem.

	\begin{itemize}
		\item  In this paper we have focused only on the result at the level of the expectation, while in practice one only has access to the empirical distribution for finite samples. How many samples do we need to take to ensure that the same results holds with high probability?

		\item  While the definition of our Riemannian gradient flow is completely canonical, any discretized numerical scheme for the gradient flow involves additional choices, such as the time step choice. What are the optimal choices of these parameters?

		\item  When the dimension $d$ is high, then $\Sigma$ has $O(d^2)$ entries, and a priori is not close to being low rank. How can one make the algorithm computationally efficient?

		\item 
		
		If we want to find all the deep vacua of $\mathcal{J}(\Sigma;\lambda)$, what is the optimal way to initialize the choice of $\Sigma$?

		\item  Is there any way to generalize our kernel learning problem to emulate the deep neural networks?

	\end{itemize}

	\begin{Acknowledgement}
		Yang Li thanks C. Letrouit, B. Geshkovski and E. Cornacchia for discussions related to machine learning. Feng Ruan thanks K. Liu for discussions related to kernel methods and their relevance to contemporary machine learning.
	\end{Acknowledgement}

\end{document}